\documentclass[a4paper,12pt]{article}
\usepackage[latin1]{inputenc}
\usepackage{amssymb}
\usepackage{latexsym}
\usepackage{amsfonts}
\usepackage{amsmath}
\usepackage{bm}
\usepackage{bbm}
\usepackage{mathrsfs}
\usepackage[T1]{fontenc}
\usepackage{ae}
\usepackage[english]{varioref}

\newcommand{\N}{\mathbbm{N}}
\newcommand{\Z}{\mathbbm{Z}}

\newcommand{\R}{\mathbbm{R}}
\providecommand{\abs}[1]{\left\lvert #1 \right\rvert}

\newcommand{\Mat}{\operatorname{Mat}}
\newcommand{\AP}{\operatorname{AP}}
\usepackage{amsthm}

\theoremstyle{plain}
\newtheorem{thm}{Theorem} 
\newtheorem*{thm*}{Theorem} 

\newtheorem{lemma}[thm]{Lemma}
\newtheorem{cor}[thm]{Corollary}

\theoremstyle{definition}
\newtheorem{defn}[thm]{Definition}

\theoremstyle{remark}

\newtheorem{remark}[thm]{Remark}
\newtheorem*{remark*}{Remark}
\newtheorem*{claim*}{Claim}

\newcommand{\refthm}[1]{Theorem~\vref{#1}}

\newcommand{\reflemma}[1]{Lemma~\vref{#1}}

\title{\bf{On arithmetic progressions in nullspaces of integer matrices}}
\author{Jonas Lindstrøm Jensen (jonas@imf.au.dk)}
\date{November 2009}

\begin{document}
\maketitle

\begin{abstract}
Inspired by the Erd\H{o}s-Turan conjecture we consider subsets of the natural numbers that contains infinitely many aritmetic progressions (APs) of any given length -- such sets will be called AP-sets and we know due to the Green-Tao Theorem and Szémeredis Theorem that the primes and all subsets of positive upper density are AP-sets. We prove that $(1,1,\ldots,1)$ is a solution to the equation
\[ M \b x = 0, \]
where $M$ is an integer matrix whose null space has dimension at least $2$, if and only if the equation has infinitely many solutions such that the coordinates of each solution are elements in the same AP. This gives us a new arithmetic characterization of AP-sets, namely that they are the sets that have infinitely many solutions to a homogeneous system of linear equations, whenever the sum of the columns is zero.
\end{abstract}

\section{Introduction}
In this paper we are studying subsets $A \subseteq \N$ that contains arbitrarily long arithmetic progressions. In this paper we give a new characterization of such sets, namely that given an integer matrix whose nullspace has dimension at least two and contains $(1,1,\ldots,1)$, we have that this nullspace contains infinitely many vectors with coordinates in $A$. This gives us a new arithmetic structure on such sets, which gives a new formulation of the Erd\H{o}s-Turan Conjecture.

Since the Green-Tao Theorem gives us that the primes is an AP-set we get as a corollary that we now have a condition for a homogeneous linear equation to have solutions with prime coordinates. Granville has already studied several additive structeres in the primes that can be derived from the Green-Tao Theorem and this paper is inspired by his work.

Several papers has been written on what linear equations have solutions with prime coordinates. Balog \cite{lit:balog} gave a lower bound on the number of prime solutions to a homogeneous system of linear equations $M \b x = 0$ if the matrix $M$ has a certain admissible structure, the null space contains a vector with positive coordinates and $M \b x \equiv 0 \pmod p^\alpha$ has integer solutions coprime to $p$ for all prime powers $p^\alpha$. In particular he proved that if $M$ is admissible and $(1,1,\ldots,1)$ is a solution, then $M \b x = 0$ has prime solutions. Choi, Liu and Tsang \cite{lit:choi} has considered upper bounds for prime solutions to ternary linear equations.

The results in this paper have been found while working on my master thesis and I would like to thank my supervisors Jørgen Brandt and Simon Kristensen for their help. I would furthermore like to thank Andrew Granville for reading and commenting on the results.

\section{APs and GAPs}
As we are considering arithmetic progressions the following notation will come in handy.
\begin{defn}[Arithmetic progressions] Let $k, d \geq 1$ and $a \geq 0$ be integers. Then an \emph{arithmetic progression} (AP) of length $k$, base $a$ and step $d$ is the set
\[ \AP(k, a, d) = \{ a + \lambda d \mid 0 \leq \lambda < k \}. \]
\end{defn}

We consider subsets of $\N$ that contains arbitrarily large arithmetic progressions. We will call these sets AP-sets and define them as follows.
\begin{defn}[AP-set] Let $A \subseteq \N$. We will call $A$ an \emph{AP-set} if there for any $k \geq 1$ exists a pair $(a, d) \in \N^2$ such that
\[ \AP(k,a,d) \subseteq A. \]
\end{defn}
\begin{remark} Notice that an AP-set contains infinitely many APs of any length. \end{remark}

We will now consider generalizes arithmetic progressions which we define as follows.
\begin{defn}[Generalized arithmetic progressions] Let $d \geq 1$, $a \geq 0$, $b_1, \ldots, b_d \geq 1$ and $N_1, \ldots, N_d \geq 1$ be integers. Then a \emph{generalized arithmetic progression} (GAP) of dimension $d$, base $a$, step $(b_1,\ldots,b_d)$ and volume $(N_1, \ldots, N_d)$ is the set
\[ \{ a + n_1 b_1 + \cdots + n_d b_d \mid 0 \leq n_i < N_i \text{ for all } i \}. \]
\end{defn}
\begin{remark} Notice that a GAP of dimension $d$, base $a$, step $(b_1, \ldots, b_d)$ and volume $(2N_1-1, \ldots, 2N_d-1)$ can be written as
\begin{equation}\label{eq:GAPcenter} \{ a' + n_1 b_1 + \cdots + n_d b_d \mid -N_i < n_i < N_i \text{ for all } i \} \end{equation}
where $a' = a + (N_1-1) b_1 + \cdots + (N_d-1) b_d$.
\end{remark}

We can construct a GAP of any dimension and volume from a sufficiently long AP, so in particular an AP-set contains infinitely many GAPs of any given dimension and volume. The following lemma is taken from \cite{lit:granville} and gives us a little more than just GAPs in AP-sets.
\begin{lemma}\label{cor:GAP} Any AP-set containts infitely many GAPs of any given dimension and volume such that each GAP is contained in an AP. \end{lemma}

\section{Finding solutions in an AP-set}
Using the existence of GAPs in AP-sets we can now find infinitely many solutions to systems of linear equations in any AP-set. To characterize this we need the following definition of what we mean by solutions in AP-sets. Recall that for a matrix $M$ we let $N(M)$ denote the nullspace of $M$.

\begin{defn} Let $M$ be an integer matrix. Then $M$ is an \emph{AP-matrix} if there is a $k \in \N$ such that for each $a,d \in \N$ there is a vector $v = (v_1,\ldots,v_n) \in N(M)$ such that 
$v_1,\ldots,v_n \in \N$ not all equal, and
\[ v_1,\ldots,v_n \in \AP(k,a,d). \]
\end{defn}

\begin{defn} Let $M$ be an integer matrix. Then $M$ is \emph{null-diagonal} if $\dim N(M) \geq 2$ and $(1,1,\ldots,1) \in N(M)$. \end{defn}

This is exactly the condition we need and we are now ready to prove the following theorem.

\begin{thm}\label{thm:systemsol} Let $M$ be an integer matrix. If $M$ is null-diagonal then it is an AP-matrix. \end{thm}
\begin{proof} Each element in $V$ can be written as
\[ m_1 \b r_1 + m_2 \b r_2 + \cdots + m_{d} \b r_{d}, \quad m_i \in \R \]
where $\b r_1 = (1,1,\ldots,1)$, $\b r_i = (r_{i1}, \ldots, r_{in}) \in \Z^n$ for $2 \leq i \leq d$ and $\b r_1, \ldots, \b r_d$ are linearly independent over $\R$. Now let $N = \max_{i,j} \abs{r_{ij}}+1$ and take a GAP of dimension $d-1$ and volume $(2N-1, \ldots, 2N-1)$. According to \reflemma{cor:GAP} we can construct GAPs of any given size such that it is contained in an AP. Now take such a GAP, and as we did in \eqref{eq:GAPcenter} we write it as
\begin{equation}\label{GAPsol} \{ a + n_1 b_1 + \cdots + n_{d-1} b_{d-1} \mid -N < n_i < N \text{ for all } i\}. \end{equation}
Now
\[ a \b r_1 + b_1 \b r_2 + \ldots + b_{d-1} \b r_{d} \]
is in the nullspace of $M$, and each coordinate is an element in the GAP given in \eqref{GAPsol}. Now assume that the solution we have found has all coordinates equal. Then it is equal to $c \b r_1$ for some $c \in \N$ so
\[ (a-c)\b r_1 + b_1 \b r_2 + \ldots + b_{d-1} \b r_{d} = 0. \]
This is not possible since $\b r_1, \cdots, \b r_{d}$ are linearly independent.
\end{proof}

This immedietly yeilds the following corrollary.

\begin{cor} Let $A$ be an AP-set, and let $M$ be null-diagonal. Then the nullspace of $M$ contains infinitely many vectors $v = (v_1,\ldots,v_n)$ such that $v_i \in A$ for all $i=1,\ldots,n$. \end{cor}

\section{Prime-like sets}
\refthm{thm:systemsol} gives us a sufficient condition to be able to find infinitely many solutions in an AP-set. Let us now examine in what way it also is a nescessary condition. To examine this we need to require a bit more from our AP-set.
\begin{defn}[Prime-like sets] A set $A \subseteq \N$ is called \emph{prime-like} if for each $\AP(k, a, d) \subseteq A$ with $k \geq 3$ we have $\gcd(a,d) = 1$. \end{defn}

Notice that the primes is prime-like because if we have a progression $\AP(k, a, d)$ in the primes, then $a$ is prime and $d$ is even.

\begin{thm} Let $A$ be a prime-like AP-set, $M \in \Mat_{m,n}(\Z)$ and $k \geq 3$. Assume that the nullspace of $M$ contains infinitely many vectors $(x_1,\ldots,x_n)$ such that for each vector there are $a,d \in \N$ such that
\[ x_1, \ldots, x_n \in AP(k, a, d) \subseteq A. \]
Then the $M$ is null-diagonal.
\end{thm}
\begin{proof} Let $1 \leq i \leq m$ be given. Assume for contradiction that $a_{i1} + \cdots + a_{in} \neq 0$. Let $\{(x_1^{(j)}, \ldots, x_n^{(j)}) \mid j \in \N\}$ be the infinitely many solutions  given in the lemma. For each $j \in \N$ there exist $b_j$ and $d_j$ such that $x_l^{(j)} = b_j + \lambda_l^{(j)} d_j$ with $0 \leq \lambda_l^{(j)} < k$ for all $l=1,\ldots,n$ since each $x_l^{(j)}$ is an element of $\AP(k, b_j, d_j)$. Inserting this in $M\b x=0$ we get that we for each $j \in \N$ have
\[ b_j(a_{i1} + \cdots + a_{in}) = -d_j(a_{i1} \lambda_1^{(j)} + \cdots + a_{in} \lambda_n^{(j)}). \]
Since $\gcd(b_j,d_j) = 1$, $b_j$ must divide $a_1 \lambda_1^{(j)} + \cdots + a_n \lambda_n^{(j)}$ so if we let $C = \abs{a_{i1}} + \cdots + \abs{a_{in}}$ we have $b_j \leq Ck$. Now
\[ \abs{d_j} = \abs{b_j \frac{a_1 + \cdots + a_n}{a_1 \lambda_1^{(j)} + \cdots + a_n \lambda_n^{(j)}}} \leq Ck \]
so the set $\{d_j \mid j \in \N\}$ is also finite. The solutions $\{(x_1^{(j)}, \ldots, x_n^{(j)}) \mid j \in \N\}$ are therefore taken from only finitely many APs of length $k$, and there can hence be only finitely many of them. This is a contradiction against the assumption, and this finishes the proof.
\end{proof}


Combining this with the corrollary of \refthm{thm:systemsol} we get the following.
\begin{thm} Let $A$ be a prime-like AP-set and let $M$ be an integer matrix. Then there is a $k \in \N$ such that the nullspace of $M$ have infinitely many vectors with all coordinates being elements of the same AP of length $k$ in $A$ if and only if $M$ is null-diagonal.
\end{thm}

We now give an example of an application of \refthm{thm:systemsol}. This is a known result, see for instance \cite{lit:granville}.
\begin{cor} Let an AP-set $A$ and $n \geq 1$ be given. Then there exists infinitely many $n$-tuples in $x_1,\ldots,x_n \in A$ with $x_i \neq x_j$ for some $i,j$ such that
\[ \frac{x_1 + \cdots + x_n}{n} \in A. \]
\end{cor}
\begin{proof} When $n = 1$ it is trivial so let $n \geq 2$ be given. Consider the linear equation
\[ x_1 + \cdots + x_n - nx_{n+1} = 0. \]
From \refthm{thm:systemsol} we know that this equation has infinitely many solutions $x_1,\ldots,x_n,x_{n+1} \in A$ with $x_i \neq x_j$ for some $i,j$. Now for each of these we have
\[ \frac{x_1 + \cdots + x_n}{n} = x_{n+1} \in A, \]
which finishes the proof.
\end{proof}

\section{Zero-solution sets}
We have proved that in any AP-set we can find infinitely many solutions to any system of linear equation, as long as the sum of the columns of the matrix is zero. This motivates the following definition.

\begin{defn}[Zero-solution sets] A set $A \subseteq \N$ is a \emph{zero-solution set} if any null-diagonal $M$ contains infinitely many vectors $\b x = (x_1, \ldots, x_n)$ with $x_1, \ldots, x_n \in A$ and $x_i \neq x_j$ for some $i,j$.
\end{defn}

Now \refthm{thm:systemsol} can be formulated as follows: If $A$ is an AP-set then $A$ is a zero-solution set. We now want to prove that zero-solution sets and AP-sets are the same.
\begin{thm} Let $A \subseteq \N$. Then $A$ is a zero-solution set if and only if $A$ is an AP-set.\end{thm}
\begin{proof} The 'if' part we get from \refthm{thm:systemsol}. Let $n \geq 3$ be an integer and let $M \in \Mat_{n-2,n}(\Z)$ be given such that the solution space of $M \b x = 0$ is given by
\[ m_1 (1,1,\ldots,1) + m_2 (0,1,2,\ldots,n-1), \quad m_1, m_2 \in \R. \]
Since $A$ is a zero-solution set there are infinitely many solutions in $A$ with $m_2 \neq 0$. We also see that such a solution is in $A$ so it is integer and both $m_1$ and $m_2$ are hence integer. Each of these solutions gives us an AP of length $n$.
\end{proof}

This result gives us a new formulation of the Erd\H{o}s-Turan conjecture \cite{lit:erdos},
\[ \sum_{a \in A} \frac{1}{a} = \infty \; \Rightarrow \; A \textrm{ is a zero-solution set}. \]

\end{document}